\documentclass[]{amsart}
\usepackage[left=1.5in,right=1.5in,top=1.5in,bottom=1.5in]{geometry}
\usepackage{tikz-cd}
\usepackage{tikz}
\usepackage{amsfonts,latexsym,rawfonts,amsmath,amssymb, amsthm,bm}
\usepackage[T1]{fontenc}
\usepackage{latexsym,lscape,rawfonts}
\usepackage{hyperref}
\usepackage{CJK}
\usepackage{mathrsfs, enumitem}
\usepackage[title]{appendix}
\usepackage{rotating}
\usepackage{graphicx}
\usepackage{galois}
\usepackage[all,cmtip]{xy}
\usepackage{extarrows,color}
\usepackage{times}

\newenvironment{thmbis}[1]
{\renewcommand{\thethm}{\ref{#1}$ $}%
	\addtocounter{thm}{-1}%
	\begin{thm}}
	{\end{thm}}

\newtheorem{thm}{Theorem}[section]

\newtheorem*{fixed point criterion}{\fixed point criterion}
\newtheorem{cor}[thm]{Corollary}
\newtheorem{lem}[thm]{Lemma}
\newtheorem{prop}[thm]{Proposition}

\theoremstyle{definition}
\newtheorem{defn}[thm]{Definition}
\newtheorem{exam}[thm]{Example}
\newtheorem{ques}[thm]{Question}
\theoremstyle{remark}
\newtheorem{rem}[thm]{Remark}
\numberwithin{equation}{section}

\newcommand{\Z}{\mathbb Z}

\newcommand{\C}{\mathcal C}
\newcommand{\A}{\mathcal A}
\newcommand{\mS}{\mathcal S}

\newcommand{\E}{\mathrm{E}}
\newcommand{\V}{\mathrm{V}}

\makeatletter
\newcommand*\bigcdot{\mathpalette\bigcdot@{1.5}}
\newcommand*\bigcdot@[2]{\mathbin{\vcenter{\hbox{\scalebox{#2}{$\m@th#1\bullet$}}}}}
\makeatother

\begin{document}
	
\title{Alternating and symmetric separability of free products}

\author{Dongxiao Zhao}
\address{School of Mathematics and Statistics, Xi'an Jiaotong University, Xi'an 710049, CHINA}
\email{zdxmath@stu.xjtu.edu.cn\\ORCID: 0009-0006-7919-9244}
	
\author{Qiang Zhang}
\address{School of Mathematics and Statistics, Xi'an Jiaotong University, Xi'an 710049, CHINA}
\email{zhangq.math@mail.xjtu.edu.cn\\ORCID: 0000-0001-6332-5476}
	
	
\thanks{The authors are partially supported by NSFC (No. 12471066) and the Shaanxi Fundamental Science Research Project for Mathematics and Physics (No. 23JSY027).}


\date{\today}


\begin{abstract}

 Let $F \ast G$ be a free product of a free group $F$ and a LERF group $G$. In this note, we provide sufficient conditions for a subgroup $H$ of $F \ast G$ to be $\A\cup \mS$-separable, that is, for any finite set $\{\gamma_1, \ldots, \gamma_n\} \subset (F \ast G) \setminus H$, there is a surjection $f$ from $F \ast G$ to an alternating or symmetric group such that $f(\gamma_i) \notin f(H)$ for all $i$. As a corollary, any finitely generated infinite-index subgroup of a free group is $\A\cup \mS$-separable in the free product of the free group and an arbitrary LERF group, generalizing a result of Wilton. 
\end{abstract}
\maketitle

    \section{Introduction}
    Residual properties have been widely discussed on the topic of infinite discrete groups. Let $\C$ be a class of groups. A group $G$ is said to be \textit{residually $\C$} if for every $g \in G\setminus \{e\}$, there exists a homomorphism $f: G \to C$ onto some group $C \in \C$, such that $f(g)$ is nontrivial. The property of residual finiteness admits several generalizations, most notably the concept of LERF:
    a group $G$ is said to be \textit{locally extended residually finite} (\textit{LERF}) if for any finitely generated subgroup $H \leq G$ and any $\gamma \in G \setminus H$, there is a homomorphism $f$ from $G$ onto a finite group such that $f(\gamma)\notin f(H)$.
   
    The LERFness of groups is a property closely related to low dimensional topology and has been widely studied. Large classes of groups are known to be LERF: for example, finitely generated free groups \cite{Ha}, closed orientable hyperbolic surface groups \cite{Sc}, compact Seifert fiber space groups \cite{Sc}, Sol manifold groups \cite{Mal} and hyperbolic 3-manifold groups \cite{Ag} \cite{Wise}. In fact, by \cite{Su}, a compact (with empty or tori boundary) orientable irreducible 3-manifold $M$ with $\pi_1(M)$ LERF if and only if $M$ supports one of Thurston's eight geometries. Moreover,
    by \cite{Wi08}, limit groups are LERF, which generalizes both the cases of free groups and surface groups. However, nontrivial graph manifold groups are not LERF \cite{NW}, noncompact arithmetic hyperbolic $m$-manifold groups with $m>3$ and compact arithmetic hyperbolic $m$-manifold groups with $m>4$ (except when defined by octonions with $m=7$) are not LERF \cite{Su}.
    
    The LERFness admits a broader generalization through the restriction of its finite quotients to a more specific class, for example, the class of alternating groups $\A$ and symmetric groups $\mS$. 
    \begin{defn}
        Let $H$ be a subgroup of a finitely generated group $G$ and let $\C$ be a class of groups. We say that $H$ is $\C$-\textit{separable} in $G$ if for any finite set $\{\gamma_1, \ldots, \gamma_n\} \subset G \setminus H$, there is a surjective homomorphism $f$ from $G$ to a group in $\C$ such that $f(\gamma_i) \notin f(H)$ for all $i$.
    \end{defn}
    Note that $\C$-separability is a property of a certain subgroup, in particular, the trivial subgroup of $G$ is $\C$-separable implies that $G$ is residually $\C$.
    
    \textbf{Notations.} Throughout this paper, let $F_r$ be a free group of rank $r>1$ and let $H^g=gHg^{-1}$ be the conjugation of a subgroup $H$ by $g$. For a set $V$, let $|V|$ denote the cardinality of $V$. 
    
    In \cite{Wi}, Wilton showed that the finitely generated infinite-index subgroups of free groups are $\A$-separable. More precisely,

    \begin{thm}\cite[Theorem A]{Wi}\label{F is LERA}
      Let $H$ be a finitely generated subgroup of infinite index in the free group $F_r$ and let $\{\gamma_1, \ldots, \gamma_n\}$ be any finite set of $F \setminus H$. Then there is a surjection $f$ from $F_r$ to some finite alternating group $A_k$ such that $f(\gamma_i) \notin f(H)$ for all $i$.
    \end{thm}

   Wilton also conjectured that the same result of Theorem \ref{F is LERA} holds for closed orientable hyperbolic surface groups. In 2021, Buran \cite{Bu} confirmed Wilton's conjecture. Moreover, he generalized this result to right-angled Coxeter groups and provided a complete characterization of the $\A$-separability and the $\mS$-separability of convex-cocompact subgroups.

    We now investigate the residual properties and the $\C$-separability of free products of groups. 
    It is well known that both of the residually finiteness and LERFness are preserved under taking free products, see \cite{Ro} and \cite{Bur}. In 1984, Tamburini and Wilson \cite{TW} studied the residually alternating property of the free products of residually finite groups. 
    \begin{thm}\cite[Theorem 1]{TW}\label{tw}
        Let $G$ be a free product of at least two nontrivial residually finite groups. Then $G$ is residually alternating unless $G$ is the infinite dihedral group $\Z/2\Z\ast\Z/2\Z$. 
    \end{thm}
    
    Note that LERF groups are residually finite, and the trivial subgroup of $G$ is $\A$-separable implies that $G$ is residually alternating. It is natural to ask the following question. 

    \begin{ques}
        Which kinds of subgroups are $\A$-separable or $\mS$-separable in free products of LERF groups?
    \end{ques}

     By combining the methods in \cite{Gi, Wi}, we show that there exists a class of $\A \cup \mS$-separable subgroups in the free product of a LERF group with a free group $F_r$ of rank $r>1$.

    \begin{thm}\label{main2}
    Let $G$ be a LERF group and let $H\leq F_r\ast G$ be a finitely generated subgroup. Then $H$ is $\A \cup \mS$-separable in $F_r\ast G$ if one of the following holds:
    \begin{enumerate}
        \item $H \cap F_r^{\gamma}=1$ for every $\gamma \in F_r \ast G$, or
        \item $H \cap F_r^{\gamma} \neq 1$ has infinite index in $F_r^{\gamma}$ for some $\gamma \in F_r \ast G$.
    \end{enumerate}
    \end{thm}

   Since every finitely generated infinite-index subgroup of $F_r$ satisfies the hypothesis in Theorem \ref{main2}, the following immediate corollary represents a generalization of Theorem \ref{F is LERA} regarding the $\A$-separability of free groups.

    \begin{cor}\label{cor}
        Any finitely generated infinite-index subgroup of $F_r$ is $\A \cup \mS$-separable in the free product $F_r \ast G$, where $G$ is a LERF group.
    \end{cor}

    Notice that $\A \cup \mS$-separable means the quotient is required to be either an alternating or a symmetric group. Since $G$ is an arbitrary LERF group, no uniform modification method seems to exist that would guarantee the quotient is precisely an alternating group.
    Moreover, the hypothesis ``infinite index" is essential. In fact, the following example shows there exist finite-index subgroups that are not alternating separable in such a free product.

    \begin{exam}
    Let $F_2=\langle a,b \mid - \rangle$, $G=\langle c \mid c^2\rangle$. We consider $$\phi: F_2 \ast G \to \Z/2\Z$$ given by $$\phi(a)=\phi(b)=1, \quad \phi(c)=0.$$
    Then $\ker\phi$ is a finitely generated normal subgroup of $F_2\ast G$ of index 2. For any surjection $f: F_2\ast G \to A_n$, $f(\ker\phi)$ is normal in $A_n$ with index at most 2.  If $n=3$ or $n\geq 5$, then $A_n$ is simple and hence $f(\ker\phi)=A_n$;  if $n=4$, then $f(\ker\phi)$ is again $A_n$ because  the only proper normal subgroup of $A_4$ is of index 3. So $\ker\phi$ is not $\A$-separable.
    \end{exam}

    To prove Theorem \ref{main2}, we will construct a relative Cayley graph of $F_r \ast G$. This graph corresponds to the 1-skeleton of a primary sheeted cover of the representation complex associated with $F_r \ast G$. The action of $F_r \ast G$ on its vertex set will then be used to demonstrate $\A \cup \mS$-separability.
    
    This paper is organized as follows. In Section \ref{pre}, we review some topological tools, such as labeled graphs and precovers, to prepare for proving the main results.  In Section \ref{Wilton}, we review the specific construction of Wilton in \cite{Wi} and bring in Jordan's theorem. In Section \ref{section4}, we introduce Markus-Epstein algorithm to provide a more concrete characterization of the structure of subgroups defined by labeled graphs. Finally, in Section \ref{proof}, we employ Wilton's construction and Markus-Epstein algorithm to prove Theorem \ref{main2}.

    \section{Preliminaries}\label{pre}
    In this section, we review some notions and properties of LERFness and topology of graphs.  Let $G=\langle X \mid R \rangle$ be a group and 
    $$X^\ast:=\{x,x^{-1}| x \in X\}.$$

    \subsection{Labeled graph and embedding}
     A \textit{graph} is a 1-dimensional CW-complex consisting of two sets $\V$ and $\mathrm{E}$ representing the 0-cells and 1-cells respectively and two functions $\mathrm{E} \to \mathrm{E}$ and $\mathrm{E} \to \V$: for each $e \in \mathrm{E}$, there is an element $\bar{e} \in \mathrm{E}$ and an element $\iota(e)\in \mathrm V$. We call \textit{edge} for $e \in \mathrm{E}$ and \textit{vertex} for $v \in \V$, the requirements of $\bar{e}$ are $\bar{\bar{e}}=e$ and $\bar{e} \neq e$. A \textit{map of graph} consists of a pair of functions preserving the structure, which maps edges to edges, vertices to vertices.
     
    A graph $\Gamma$ is called a \textit{labeled graph} (with the group $G=\langle X \mid R \rangle$) if it is equipped with a \textit{labeling function} 
         $$\mathrm{Lab}: \mathrm{E}(\Gamma) \to X^\ast$$ satisfying $\mathrm{Lab}(\bar{e})=\mathrm{Lab}(e)^{-1}$ and the \textit{immersion condition}: $e_1=e_2$ provided $\mathrm{Lab}(e_1)=\mathrm{Lab}(e_2)$ and $\iota(e_1)=\iota(e_2)$. The \textit{label} of a path $p=e_1 e_2\cdots e_n \subset \Gamma$ is defined by 
         $$\mathrm{Lab}(p) \equiv \mathrm{Lab}(e_1)\cdot \mathrm{Lab}(e_2)\cdots\mathrm{Lab}(e_n)\in G.$$
    If in addition, any path $p$ with $\mathrm{Lab}(p)=1$ is closed, then $\Gamma$ is called $G$-\textit{based}. For each vertex $v$ and each word $g \equiv x_1 x_2 \cdots x_n$ , since $\Gamma$ is a labeled graph, there is a unique path $p$ in $\Gamma$ originating from $v$ with $\mathrm{Lab}(p) \equiv x_1 x_2 \cdots x_n$. The subgroup
$$\mathrm{Lab}(\Gamma, v):=\{\mathrm{Lab}(p) \mid  p~ \textrm{is a loop in} ~\Gamma \ \textrm{with base point} ~v\}\leq G$$
is called \textit{the subgroup of $G$ determined by the graph $\Gamma$.}

   Let $x \in X^\ast$. A vertex $v \in \V(\Gamma)$ is called $x$-\textit{saturated} if there exists $e\in \mathrm{E}(\Gamma)$ with $\iota(e)=v$ and $\mathrm{Lab}(e)=x$. If $\Gamma$ is $x$-saturated for any $x \in X^\ast$ at any $v \in \V(\Gamma)$, then $\Gamma$ is called $X^\ast$-\textit{saturated}.
    
    A map $\pi: \Gamma \to \Delta$ of two labeled graphs is an \textit{immersion} if it is injective on the set $\{e \in \mathrm{E}(\Gamma)| \iota(e)=v\}$ of all vertices of $\Gamma$. In addition, if $\pi$ preserves the labels of edges, then it is called a \textit{morphism of labeled graphs}. If $\pi$ is again an injective morphism, then it is called an \textit{embedding}  or $\Gamma$ \textit{embeds} in the graph $\Delta$. 

    \begin{rem}
    A labeled graph is also called a \textit{well-labeled graph} in some literature, see \cite{Ma1} \cite{Ma2}. Any labeled graph of a free group $F_r$ is $F_r$-based. 
    \end{rem}
    
    \subsection{Relative Cayley graph}
    For a group $G=\langle X \mid R \rangle$, we denote the Cayley graph of $G$ with respect to the given generating set $X$ by $\mathrm{Cay}(G)$. Recall that $\mathrm{Cay}(G)$ is an oriented graph with the vertex set $\V(\mathrm{Cay}(G))=G$ and the edge set $\mathrm{\mathrm{E}}(\mathrm{Cay}(G))=G \times X^\ast$. An edge $(g,x) \in \mathrm{E}(\mathrm{Cay}(G))$ originates at $g$ and terminates at $gx$. 
    
    For any subgroup $H\leq G$, we denote by $\mathrm{Cay}(G,H)$ the Cayley graph of $G$ relative to the subgroup $H$ of $G$, which is an oriented graph with vertex set consisting of all the right cosets of $H$:
    $$\V(\mathrm{Cay}(G,H)):=\{Hg\mid g\in G\},$$ 
    and the edge set $\mathrm{E}(\mathrm{Cay}(G,H))=\V(\mathrm{Cay}(G,H)) \times X^\ast$.  We use $\mathrm{Cay}(G,H,H \cdot 1)$ when the base point $H \cdot 1$ is emphasized. An edge $(Hg,x)$ originates from the vertex $Hg$ and ends at $Hgx$. Moreover, a path $p$ originating from $Hg$ must end at $Hg\mathrm{Lab}(p)$. Therefore, we have the following.
    \begin{lem}\label{loop lem in relative Cayley graph}
        Let $H\leq G$ be a subgroup of $G$. A path $p$ originating from the base point $H\cdot 1$ in $\mathrm{Cay}(G,H, H \cdot 1)$ is a loop if and only if $\mathrm{Lab}(p)\in H$. Hence, $$\mathrm{Lab}(\mathrm{Cay}(G,H), H\cdot 1)=H.$$
    \end{lem}

    The following lemma shows that $G$-based graphs are subgraphs of the relative Cayley graphs.
    \begin{lem}\cite[Lemma 1.5]{Gi}\label{subgraphs of Cayley}
        Let $G=\langle X \mid R \rangle$ be a group and let $(\Gamma,v_0)$ be a graph labeled with $X$. Denote $\mathrm{Lab}(\Gamma,v_0)=H$. Then
        \begin{enumerate}
            \item $\Gamma$ is $G$-based if and only if it can be embedded in $\mathrm{Cay}(G,H,H \cdot 1)$;
            \item $\Gamma$ is $G$-based and $X^\ast$-saturated if and only if it is isomorphic to $\mathrm{Cay}(G,H,H \cdot 1)$.
        \end{enumerate}
    \end{lem}

  Note that $\mathrm{Cay}(G)$ and $\mathrm{Cay}(G,H)$ are naturally labeled with the generating set $X$ of $G$, and
    there are natural action of $G$ on $\V(\mathrm{Cay}(G))$ and $\V(\mathrm{Cay}(G, H))$ given by left translation. The relative Cayley graph $\mathrm{Cay}(G,H)$ can be defined as the quotient of $\mathrm{Cay}(G)$ by the action of $H$. Moreover, $\mathrm{Cay}(G,H)$ can be viewed as the 1-skeleton of a covering space of the representation complex of $G$, so we follow the terminology of Gitik \cite{Gi} and refer to the "\textit{cover}" as the \textit{relative Cayley graph}.
        
    Scott \cite{Sc} described the residually finite and LERF properties by the embeddability from subgraphs to covers. Gitik \cite{Gi} rephrased Scott's result as follows.
    
    \begin{thm}\cite[Theorem 1.1]{Gi}\label{scott}
        Let $G=\langle X \mid R \rangle$ be a group.
        \begin{enumerate}
            \item $G$ is residually finite if and only if for any finite tree $\Gamma$, which is a subgraph of the Cayley graph of $G$, there exists a finite-index subgroup $H_0$ of G such that $\Gamma$ can be embedded in $\mathrm{Cay}(G,H_0)$.

             \item $G$ is LERF if and only if for any finitely generated subgroup $H$ of $G$ and for any finite connected subgraph $(\Gamma,H \cdot 1)$ of $\mathrm{Cay}(G,H, H \cdot 1)$, there exists a finite-index subgroup $H_0$ of $G$ such that $(\Gamma,H \cdot 1)$ can be embedded in $\mathrm{Cay}(G,H_0,H_0 \cdot 1)$.
        \end{enumerate}
    \end{thm}
    
    \subsection{Precover}
    The notion of precover was defined by Gitik in \cite{Gi} and then actively employed by Markus-Epstein to study algorithmic problems in amalgams of finite groups in \cite{Ma1, Ma2}.

    Let $G_1$ and $G_2$ be two groups with generating sets $X$ and $Y$ respectively such that $X^\ast \cap Y^\ast = \emptyset$, and let $A=G_1 \mathop{\ast} G_2$ be the free product of $G_1$ and $G_2$. 
    \begin{defn}
    Let $\Gamma$ be a graph labeled with $X \cup Y$.
        \begin{enumerate}
            \item A subgraph of $\Gamma$ is called \textit{monochromatic} if it is labeled only with $X$ or only with $Y$. An $X$-\textit{component} (resp. $Y$-\textit{component}) of $\Gamma$ is a maximal connected subgraph of $\Gamma$ labeled with $X$ (resp. $Y$).
       
        \item A vertex $v$ in $\V(\Gamma)$ is called $X$-\textit{monochromatic} (resp. $Y$-\textit{monochromatic}), if all edges connected to $v$ are labeled with $X$ (resp. $Y$). Otherwise, $v$ is called \textit{bichromatic}.

        \item $\Gamma$ is called a \textit{precover} of $A$, if it is $A$-based, and each $X$-component of $\Gamma$ is a cover of $G_1$ and each $Y$-component of $\Gamma$ is a cover of $G_2$.
         \end{enumerate}
    \end{defn}

    There is an equivalent description of the $A$-based graph as follows. 
    \begin{lem}[\cite{Gi}]\label{free product based}
        If $A=G_1 \ast G_2$ is a free product, then any graph labeled with $X \cup Y$ is $A$-based if and only if each $G_i$-component is $G_i$-based for $i=1,2$. In particular, it is $A$-based if each monochromatic component is a cover of the corresponding free factor.
    \end{lem}

    \begin{lem}\label{p in component}
       Let $\Gamma$ be a finite, connected, $A$-based graph labeled with $X \cup Y$ and $\Gamma_1$ be an $X$-component. Let $P$ be a loop in $\Gamma$ starting from $v_1 \in \Gamma_1$. If $1 \neq \mathrm{Lab}(P) \in G_1$, then there is a loop $P'$ in $\Gamma_1$ starting from $v_1$ such that $\mathrm{Lab}(P')=\mathrm{Lab}(P)$.
    \end{lem}
    \begin{proof}
        We denote $P = P_1P_2 \cdots P_n$, each $P_j(1\leq j\leq n)$ is a maximal connected subpath labeled only with $X$ or $Y$. Then $\mathrm{Lab}(P_j) \in G_i$ for some $i\in {1, 2}$. Moreover, since $$\mathrm{Lab}(P) =\mathrm{Lab}(P_1)\mathrm{Lab}(P_2)\cdots \mathrm{Lab}(P_n)  \in G_1$$ and $A$ is a free product of $G_1$ and $G_2$, we have $\mathrm{Lab}(P_i)=1$ for each $P_i$ labeled with $Y$. Then $P_i$ must be closed because $\Gamma$ is $A$-based. Moreover, as $P_i$ is maximal, both $P_{i-1}$ and $P_{i+1}$ are contained in $X$-components with 
        $$\tau(P_{i-1})=\iota(P_{i})=\tau(P_{i})=\iota(P_{i+1}).$$ It implies $P_{i-1}$ and $P_{i+1}$ are contained in the same $X$-component. So we can remove $P_i$ and combine $P_{i-1}$ with $P_{i+1}$ as a new monochromatic component. After removing all such $P_i$, we obtain a loop $P'$ that starts from $v_1$ and is contained in the $X$-component $\Gamma_1$, such that $\mathrm{Lab}(P')=\mathrm{Lab}(P)$.
    \end{proof}
    
    The precover of free products has a key property. 
    \begin{lem}\cite[Lemma 2.4]{Gi}\label{embedding of precover}
        For any groups $G_1=\langle X \mid R \rangle$ and $G_2=\langle Y \mid T \rangle$, each precover of $G_1 \ast G_2$ can be embedded in a cover of $G_1\ast G_2$ with the same set of vertices.
    \end{lem}

 \subsection{Pushout and folding}

    Pushout is an effective way to construct a new graph from old graphs considered by Stallings \cite{St}. For example, in the category of labeled graphs, if $\alpha_1:\Gamma \to \Delta_1$ and $\alpha_2:\Gamma \to \Delta_2$ are injective maps, then the amalgam of $\Delta_1$ and $\Delta_2$ denoted by $\Delta_1 \mathop{\ast}\limits_{\Gamma} \Delta_2$ is given by the pushout of $\alpha_1$ and $\alpha_2$:
    $$\begin{tikzcd}
     \Gamma  \arrow[r, "\alpha_1"] \arrow[d, "\alpha_2"]  &  \Delta_1 \arrow[d] \\
     \Delta_2 \arrow[r]  &  \Delta_1 \mathop{\ast}\limits_{\Gamma} \Delta_2
    \end{tikzcd}$$
    In general, none of the graphs need be connected and none of the maps need be injective, moreover, pushouts do not always exist. The pushout of $\alpha_1:\Gamma \to \Delta_1$ and $\alpha_2:\Gamma \to \Delta_2$ exists if and only if there are orientations of $\Gamma$, $\Delta_1$ and $\Delta_2$ which are preserved by $\alpha_1$ and $\alpha_2$. We only consider the case that both $\alpha_1$ and $\alpha_2$ are injective maps, the amalgam can be constructed by performing the foldings.

   Let $f: \mathrm{E}(\Gamma) \to X^\ast$ be a function on the graph $\Gamma$ such that $f(\bar{e})=(f(e))^{-1}$. For any two distinct edges $e_1, e_2\in \mathrm{E}(\Gamma)$ with $\iota(e_1)=\iota(e_2)$ and $e_1 \neq \bar{e}_2$, if $f(e_1)=f(e_2)$, then the \textit{folding} $(e_1,e_2)$ is the projection of $\Gamma$ onto the quotient graph $\Gamma_1$ which is constructed by identifying $e_1$ with $e_2$ and $\tau(e_1)$ with $\tau(e_2)$. After performing a finite sequence of foldings as much as possible, the corresponding graph of $\Gamma$ is indeed a labeled graph. It can be easily seen that the above amalgam of $\alpha_1$ and $\alpha_2$ can be constructed by taking the disjoint union of graphs, identifying $\alpha_1(\Gamma)$ with $\alpha_2(\Gamma)$ and then performing finite foldings until a labeled graph is obtained.

    \section{Wilton's proof of free groups}\label{Wilton}
    In this section, we recall the specific structure of Wilton's proof of that any finitely generated infinite-index subgroup of a free group with rank greater than one is $\A$-separable.

    \subsection{Jordan's theorem}
An action of a group $H$ on  $\{1,2,\ldots,n\}$ is called \textit{symmetric} (resp. \textit{alternating}) if the image of $H\to S_n$ is $S_n$ (resp. $A_n$).   A subgroup $G \leq S_n$ is called \textit{transitive} if it acts transitively on $\{1,2,\ldots,n\}$, and called \textit{primitive} if in addition, it does not preserve any nontrivial partition.

    The main tool for proving Lemma \ref{alternating embedding} is the following Jordan's theorem,  which gives a criterion for an action to be symmetric or alternating.
    
    \begin{thm}[Jordan's theorem, \cite{DM}]\label{Jordan}
        For each $k > 2$ there exists $N$ such that if $n > N$, $G \leq S_n$ is a primitive subgroup and there exists $\gamma \in G \setminus \{e\} $, which moves less than $k$ elements, then $G=S_n$ or $A_n$.
    \end{thm}

   Note that if $G \leq S_n$ is transitive, then the cardinality of any partition preserved by $G$ divides $n$. Therefore, if $p$ is prime,  then $G\leq S_p$ is primitive if and only if $G$ is transitive. We have an immediate corollary of Theorem \ref{Jordan} as follows. 

    \begin{lem}\label{Jordan for p}
        For each $k > 2$, there exists $N$ such that for any prime $p>N$ and any transitive subgroup $G \leq S_p$ if there exists $\gamma \in G \setminus \{e\} $, which moves less than $k$ elements, then $G=S_p$ or $A_p$.
    \end{lem}

    \subsection{Finite cover and Wilton's construction}
    
    Let $F_r$ be a free group of rank $r>1$ with the generating set $X=\{x_1,\cdots,x_r\}$. Stallings \cite{St} showed: \textit{any finite graph labeled with $X$ can be embedded into a cover without increasing the number of vertices}. Wilton made a modification in the construction to force the action of $F_r$ on the vertex set of the covering to be alternating or symmetric, see \cite[Lemma 18 and Lemma 19]{Wi}. We summarize them to the following lemma and describe Wilton's specific construction in detail as a proof, which shall be used in our proof of Theorem \ref{main2}.
    \begin{lem}\label{alternating embedding}
        Let $Z$ be a finite, connected graph labeled with $X$.
        If $Z$ is not $X^\ast$-saturated, then there is a finite cover $Y$ of $F_r$ such that $Z$ embeds in $Y$ and the action of $F_r$ on $\V(Y)$ is alternating or symmetric.
    \end{lem}
    
    \begin{proof}
        Suppose $Z$ has $k$ vertices. Then by Lemma \ref{Jordan for p}, there exists a sufficiently large prime $p$ for $k+5$ such that for any transitive subgroup $G\leq S_p$ if there exists $\gamma \in G \setminus \{e\} $, which moves less than $k+5$ elements, then $G=S_p$ or $A_p$. We shall construct a finite cover $Y$ with $p$ vertices, such that $F_r$ acts on $\V(Y)$ transitively and there is an element of $F_r$ moving less than $k+5$ vertices.
       
        Since $Z$ is not an $X^\ast$-saturated graph, there are (not necessarily distinct) vertices $a,b$ and some $x \in X$ such that $Z$ is not $x$-saturated at $a$ and not $x^{-1}$-saturated at $b$. Without loss of generality, we take $x=x_1$.

       We begin by constructing two particular labeled graphs $W_n$ and $V_s$ with $n=p-k-4$ and $4$ vertices respectively, such that $x_i\in X$ acts trivially on $\V(W_n)$ for each $i \geq 2$ and $x_2$ acts nontrivially on $\V(V_s)$. 
        
    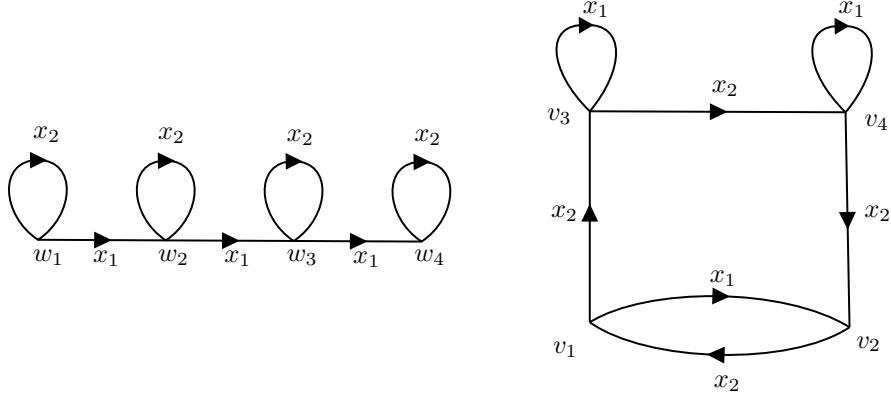
\begin{figure}
    \tikzset{every picture/.style={line width=0.75pt}} 

\begin{tikzpicture}[x=0.75pt,y=0.75pt,yscale=-1,xscale=1]

\draw    (54.73,161.86) -- (119.21,162.16) ;
\draw [shift={(91.97,162.04)}, rotate = 180.27] [fill={rgb, 255:red, 0; green, 0; blue, 0 }  ][line width=0.08]  [draw opacity=0] (8.93,-4.29) -- (0,0) -- (8.93,4.29) -- cycle    ;
\draw    (119.21,162.16) -- (183.69,162.46) ;
\draw [shift={(156.45,162.34)}, rotate = 180.27] [fill={rgb, 255:red, 0; green, 0; blue, 0 }  ][line width=0.08]  [draw opacity=0] (8.93,-4.29) -- (0,0) -- (8.93,4.29) -- cycle    ;
\draw    (183.69,162.46) -- (248.16,162.76) ;
\draw [shift={(220.93,162.64)}, rotate = 180.27] [fill={rgb, 255:red, 0; green, 0; blue, 0 }  ][line width=0.08]  [draw opacity=0] (8.93,-4.29) -- (0,0) -- (8.93,4.29) -- cycle    ;
\draw    (54.73,161.86) .. controls (36.07,149.64) and (34.36,121.63) .. (54.65,121.63) .. controls (74.93,121.63) and (73.22,149.64) .. (54.73,161.86) -- cycle ;
\draw  [fill={rgb, 255:red, 0; green, 0; blue, 0 }  ,fill opacity=1 ] (58.1,121.65) -- (51.24,125.09) -- (51.15,118.27) -- cycle ;
\draw    (119.21,162.16) .. controls (100.55,149.94) and (98.84,121.93) .. (119.12,121.93) .. controls (139.41,121.93) and (137.7,149.94) .. (119.21,162.16) -- cycle ;
\draw  [fill={rgb, 255:red, 0; green, 0; blue, 0 }  ,fill opacity=1 ] (122.57,121.95) -- (115.72,125.39) -- (115.63,118.57) -- cycle ;

\draw    (183.69,162.46) .. controls (165.03,150.24) and (163.31,122.23) .. (183.6,122.23) .. controls (203.89,122.23) and (202.17,150.24) .. (183.69,162.46) -- cycle ;
\draw  [fill={rgb, 255:red, 0; green, 0; blue, 0 }  ,fill opacity=1 ] (187.05,122.25) -- (180.2,125.69) -- (180.1,118.87) -- cycle ;

\draw    (248.16,162.76) .. controls (229.5,150.54) and (227.79,122.53) .. (248.08,122.53) .. controls (268.36,122.53) and (266.65,150.54) .. (248.16,162.76) -- cycle ;
\draw  [fill={rgb, 255:red, 0; green, 0; blue, 0 }  ,fill opacity=1 ] (251.53,122.55) -- (244.67,125.99) -- (244.58,119.17) -- cycle ;

\draw (81.05,165.66) node [anchor=north west][inner sep=0.75pt]   [align=left] {$x_1$};
\draw (147.28,165.66) node [anchor=north west][inner sep=0.75pt]   [align=left] {$x_1$};
\draw (212.21,166.29) node [anchor=north west][inner sep=0.75pt]   [align=left] {$x_1$};
\draw (50.84,103.27) node [anchor=north west][inner sep=0.75pt]   [align=left] {$x_2$};
\draw (113.84,103.9) node [anchor=north west][inner sep=0.75pt]   [align=left] {$x_2$};
\draw (178.78,104.53) node [anchor=north west][inner sep=0.75pt]   [align=left] {$x_2$};
\draw (243.07,104.53) node [anchor=north west][inner sep=0.75pt]   [align=left] {$x_2$};
\draw (50.84,165.53) node [anchor=north west][inner sep=0.75pt]   [align=left] {$w_1$};
\draw (113.84,165.53) node [anchor=north west][inner sep=0.75pt]   [align=left] {$w_2$};
\draw (178.78,165.53) node [anchor=north west][inner sep=0.75pt]   [align=left] {$w_3$};
\draw (243.07,165.53) node [anchor=north west][inner sep=0.75pt]   [align=left] {$w_4$};

\draw    (333,203.5) .. controls (362,224) and (434,226) .. (464,206) ;
\draw [shift={(392.04,219.71)}, rotate = 2.18] [fill={rgb, 255:red, 0; green, 0; blue, 0 }  ][line width=0.08]  [draw opacity=0] (8.93,-4.29) -- (0,0) -- (8.93,4.29) -- cycle    ;
\draw    (462,97.5) -- (464,206) ;
\draw [shift={(463.09,156.75)}, rotate = 268.94] [fill={rgb, 255:red, 0; green, 0; blue, 0 }  ][line width=0.08]  [draw opacity=0] (8.93,-4.29) -- (0,0) -- (8.93,4.29) -- cycle    ;
\draw    (333,97) -- (462,97.5) ;
\draw [shift={(402.5,97.27)}, rotate = 180.22] [fill={rgb, 255:red, 0; green, 0; blue, 0 }  ][line width=0.08]  [draw opacity=0] (8.93,-4.29) -- (0,0) -- (8.93,4.29) -- cycle    ;
\draw    (333,203.5) .. controls (362,185.5) and (431,185.5) .. (464,206) ;
\draw [shift={(403.76,190.51)}, rotate = 181.82] [fill={rgb, 255:red, 0; green, 0; blue, 0 }  ][line width=0.08]  [draw opacity=0] (8.93,-4.29) -- (0,0) -- (8.93,4.29) -- cycle    ;
\draw    (333,97) .. controls (351,75.5) and (351,52.5) .. (331,53.5) .. controls (310,52.5) and (312,76.5) .. (333,97) -- cycle ;
\draw  [fill={rgb, 255:red, 0; green, 0; blue, 0 }  ,fill opacity=1 ] (334,53.46) -- (328.04,56.95) -- (327.95,50.14) -- cycle ;

\draw    (462,97.5) .. controls (480,76) and (480,53) .. (460,54) .. controls (439,53) and (441,77) .. (462,97.5) -- cycle ;
\draw  [fill={rgb, 255:red, 0; green, 0; blue, 0 }  ,fill opacity=1 ] (463,53.96) -- (457.04,57.45) -- (456.95,50.64) -- cycle ;

\draw    (333,97) -- (333,203.5) ;
\draw [shift={(333,143.75)}, rotate = 90] [fill={rgb, 255:red, 0; green, 0; blue, 0 }  ][line width=0.08]  [draw opacity=0] (8.93,-4.29) -- (0,0) -- (8.93,4.29) -- cycle    ;

\draw (310,94) node [anchor=north west][inner sep=0.75pt]   [align=left] {$v_3$};
\draw (470,97) node [anchor=north west][inner sep=0.75pt]   [align=left] {$v_4$};
\draw (313,212) node [anchor=north west][inner sep=0.75pt]   [align=left] {$v_1$};
\draw (466,209) node [anchor=north west][inner sep=0.75pt]   [align=left] {$v_2$};
\draw (312,143) node [anchor=north west][inner sep=0.75pt]   [align=left] {$x_2$};
\draw (470,143) node [anchor=north west][inner sep=0.75pt]   [align=left] {$x_2$};
\draw (393,80) node [anchor=north west][inner sep=0.75pt]   [align=left] {$x_2$};
\draw (392,176) node [anchor=north west][inner sep=0.75pt]   [align=left] {$x_1$};
\draw (394,229) node [anchor=north west][inner sep=0.75pt]   [align=left] {$x_2$};
\draw (328,40) node [anchor=north west][inner sep=0.75pt]   [align=left] {$x_1$};
\draw (457,40) node [anchor=north west][inner sep=0.75pt]   [align=left] {$x_1$};

\end{tikzpicture}

\caption{The graph $W_4$ on the left and $V_{+1,-1}$ on the right.}\label{figure1}
\end{figure}

The construction of $W_n$ is as follows. We first take an interval with $n$ vertices, denoted by $w_1,\ldots,w_n$ and denote the $n-1$ edges by $e_1,\ldots,e_{n-1}$. We label all edges with $x_1$ and orient these edges by 
        $$\iota(e_1)=w_1, \quad \tau(e_j)=\iota(e_{j+1})=w_{j+1}, \quad \tau(e_{n-1})=w_n.$$
        Then for each generator $x_i \in X(i \geq 2)$ and each vertex $w_j(1 \leq j \leq n)$, we attach an oriented edge labeled with $x_i$ from $w_j$ to itself. The resulting labeled graph is $W_n$. (See the left side of Figure \ref{figure1}.) For each $i \geq 2$, $x_i$ fixes all vertices of $W_n$, and hence acts trivially on $\V(W_n)$.
        Note that the unsaturated vertices in $W_n$ are $w_1$ and $w_n$: $w_1$ is not $x_1^{-1}$-saturated, and $w_n$ is not $x_1$-saturated. 

        The construction of $V_s$ is as follows. Here, $s$ denotes an $r$-tuple $(s_i) \in \{\pm 1\}(1 \leq i \leq r)$. We first take four vertices, denoted by $v_1,v_2,v_3,v_4$ and attach an edge labeled with $x_1$ from $v_1$ to $v_2$. For each $i>2$, we attach a loop labeled with $x_i$ to both $v_1$ and $v_2$. For each $i \neq 2$, if $s_i=1$, we attach a loop labeled with $x_i$ to both $v_3$ and $v_4$; if $s_i=-1$, we attach two edges both labeled with $x_i$, from $v_3$ to $v_4$ and from $v_4$ to $v_3$, respectively. If $s_2=+1$, we attach four edges labeled with $x_2$, from $v_1$ to $v_3$, from $v_3$ to $v_1$, from $v_2$ to $v_4$ and from $v_4$ to $v_2$, respectively; if $s_2=-1$, we attach a loop of length four with vertices $v_1,v_2,v_3$ and $v_4$, each edge is labeled with $x_2$. (See the right side of Figure \ref{figure1}.) The resulting labeled graph is $V_s$ and $x_2$ acts non trivially on $\V(V_s)$.
        Note that the unsaturated vertices in $V_s$ are $v_1$ and $v_2$: $v_1$ is not $x_1^{-1}$-saturated, and $v_2$ is not $x_1$-saturated.
        
        We now construct the desired $Y$ by attaching some edges to $Z\cup W_n\cup V_s$ as follows. First, for the six unsaturated vertices $a$, $b$ in $Z$, $w_1$, $w_{n}$ in $W_{n}$, and $v_1$, $v_2$ in $V_s$, we attach three edges labeled with $x_1$: one from $a$ to $w_1$, one from $w_{n}$ to $v_1$ and one from $v_2$ to $b$. The resulting graph $Y'$ has $k+n+4=p$ vertices and it is a labeled graph. Finally, by adding some edges to $Y'$, we obtain the finite cover $Y$ without adding new vertices: 
        $$|\V(Y)|=|\V(Y')|=p.$$
        Note that $Y$ is connected and the vertices $w_1,\ldots,w_{n}$ in $Y$ are fixed by the action of $x_2$. It follows that $F_r$ acts on $\V(Y)$ transitively and $x_2$ moves less than $p-n+1=k+5$ vertices. Therefore, $Y$ meets the requirements specified in the first paragraph and the proof is complete.
\end{proof}

    \begin{rem}
        Note that to enable connecting $W_n$ to $Z$, we require the existence of an unsaturated vertex on $Z$, meaning that $Z$ is not a cover of $F_r$. If there are multiple monochromatic components, we correspondingly require that a certain $F_r$-component is not a cover.
    \end{rem}

   \section{Kurosh decomposition structure derived from the precover}\label{section4}

   In this section, we introduce the Markus-Epstein algorithm (see \cite{Ma2}), which provides a method for constructing subgroup's Kurosh decompositions via precovers. We employ this approach to characterize the structure of subgroups defined by labeled graphs.

      \begin{thm}[Kurosh Subgroup Theorem, \cite{Ku}]\label{kurosh}
          Suppose $G = \ast_{i \in I} G_i$ is a free product and $H$ is a subgroup of $G$. Then $H\cong (\ast H_j)\ast F$, where each $H_j$ is isomorphic to an intersection $H \cap G^{k_i}_i$ of $H$ with a conjugate of some $G_i$. Further, the set $\{H_j\}$ is unique up to conjugation and re-indexing, and the rank of the free group $F$ is uniquely determined.
      \end{thm}  
 
      Let $G=G_1 \ast G_2$ be the free product of $G_1$ and $G_2$ with generating sets $X$ and $Y$ respectively. Let $\Gamma$ be a finite graph labeled with $X\cup Y$ and $v_0$ be the base point of $\Gamma$. Furthermore, we assume that $\Gamma$ is connected, $G$-based and denote $$\mathrm{Lab}(\Gamma,v_0)= H \leq G.$$ 
      
      For a monochromatic component $C$ of $\Gamma$ with base point $v \in \V(C)$, let $T(C)$ be a spanning tree of $C$ with root vertex $v$. Since $\Gamma$ is connected, there is an \textit{approach path} $P_v$ from $v_0$ to $v$, we denote 
      $$\mathrm{Lab}(P_v) \equiv g_v.$$
      Without loss of generality, we can assume that $P_v \cap C=\{v\}$. (Otherwise, since $C$ is connected, we can assume the subpath $P_c=P_v \cap C$ is connected, and then change $v$ to be the initial point of $P_c$ and remove $P_c$ from $P_v$.)  If $v_0 \in \V(C)$, we choose $v=v_0$. Let $\Gamma'$ be a connected subgraph of $\Gamma$ obtained by removing edges in $\E(C)\setminus \E(T(C))$. More precisely, $\V(\Gamma')=\V(\Gamma)$ and 
$$\E(\Gamma')=(\E(\Gamma)\setminus \E(C)) \cup \E(T(C)).$$
   
   Following the above assumption, Markus-Epstein's algorithm is based on the following \cite[Lemma 4.2]{Ma2}.
      \begin{lem}\label{free factor}
      $H=\langle g_v \mathrm{Lab}(C,v) g_v^{-1}, \mathrm{Lab}(\Gamma', v_0)\rangle=g_v \mathrm{Lab}(C,v) g_v^{-1}\ast \mathrm{Lab}(\Gamma', v_0)$.
      \end{lem}

      Lemma \ref{free factor} shows that every monochromatic component $C$ with nontrivial $\mathrm{Lab}(C,v)$ corresponds to a free factor $g_v \mathrm{Lab}(C,v) g_v^{-1}$ of $\mathrm{Lab}(\Gamma,v_0)$. Then we obtain the following free decomposition of $\mathrm{Lab}(\Gamma,v_0)$.

      \begin{lem}\label{algorithm}
          Let $\Gamma$ be a finite, connected, $G$-based graph labeled with $X\cup Y$. If there are finitely many monochromatic components $C_i(1 \leq i \leq k)$ which are not trees, then 
          $$\mathrm{Lab}(\Gamma,v_0)=\big(\ast_{1\leq i\leq k}\mathrm{Lab}(P_{i}) \mathrm{Lab}(C_i,v_i) \mathrm{Lab}(P_{i})^{-1}\big) \ast \mathrm{Lab}(\Delta,v_0),$$
          where $P_{i}$ is an approach path from the base point $v_0$ of $\Gamma$ to the base point $v_i$ of $C_i$, the subgraph $\Delta$ is obtained from $\Gamma$ by removing all edges in each component $C_i$ that do not belong to its spanning tree $T(C_i)$, and $\mathrm{Lab}(\Delta,v_0)$ is a free group that does not conjugate into $G_1$ or $G_2$. 
      \end{lem}

      \begin{proof}
          Let $\mathrm{Lab}(P_{i})=g_{i}$. We proceed by induction on the number $k$ of monochromatic components which are not trees. Since there are finitely many such monochromatic components, the process is finite. When $k=1$, by Lemma \ref{free factor}, 
          $$\mathrm{Lab}(\Gamma,v_0)=g_{1} \mathrm{Lab}(C_1,v_1) g_{1}^{-1}\ast \mathrm{Lab}(\Gamma', v_0),$$
          $\Gamma'$ is the required $\Delta$.

          When $k>1$, we start from the monochromatic component $C_1$ containing the base point $v_0$ of $\Gamma$ such that $v_1=v_0 \in \V(C_1)$, then by Lemma \ref{free factor},
          $$\mathrm{Lab}(\Gamma,v_0)=g_{1} \mathrm{Lab}(C_1,v_1) g_{1}^{-1}\ast \mathrm{Lab}(\Gamma'', v_0).$$
          Since $\Gamma''$ embeds in $\Gamma$ and has fewer monochromatic components which are not trees than $\Gamma$, by the induction hypothesis, we have
          $$\mathrm{Lab}(\Gamma'',v_0)=(\ast_{2\leq i\leq k} ~g_i \mathrm{Lab}(C_i,v_i) g_i^{-1}) \ast \mathrm{Lab}(\Delta,v_0).$$
          By \cite[Lemma 5.2]{Ma2}, $\mathrm{Lab}(\Delta,v_0)$ is a free group that does not conjugate into $G_1$ or $G_2$.
          Then combining the above two equation, the proof is complete.
      \end{proof}

      Moreover, the free decomposition in Lemma \ref{algorithm} coincides with the standard Kurosh decomposition.

      \begin{lem}\label{intersection}
         In Lemma \ref{algorithm}, each free factor 
         $$L_i:=\mathrm{Lab}(P_{i}) \mathrm{Lab}(C_i,v_i) \mathrm{Lab}(P_{i})^{-1}=\mathrm{Lab}(\Gamma,v_0) \cap G_k^{\mathrm{Lab}(P_{i})}$$ 
         for some $k \in \{1,2\}$ and $\mathrm{Lab}(C_i,v_i) \leq G_k$.
      \end{lem}
      \begin{proof}
           Since $C_i$ is a monochromatic component, $\mathrm{Lab}(C_i,v_i)\leq G_k$ for some $k \in \{1,2\}$ and
           $$L_i \leq \mathrm{Lab}(\Gamma,v_0) \cap G_k^{\mathrm{Lab}(P_i)}$$ clearly. Below we show $ \mathrm{Lab}(\Gamma,v_0) \cap G_k^{\mathrm{Lab}(P_i)}\leq L_i.$

           For any $\gamma \in \mathrm{Lab}(\Gamma,v_0)\cap G_k^{\mathrm{Lab}(P_i)}$, we can denote 
           $\gamma =\mathrm{Lab}(P_i) g \mathrm{Lab}(P_i)^{-1}$ for some $g\in G_k$. Since $\gamma \in \mathrm{Lab}(\Gamma,v_0)$, there is a loop $P_\gamma$ starting from $v_0$ with $\gamma = \mathrm{Lab}(P_{\gamma})$. It implies that 
           $$g = \mathrm{Lab}(P_i)^{-1} \mathrm{Lab}(P_{\gamma})\mathrm{Lab}(P_i)=\mathrm{Lab}(\bar{P}_iP_{\gamma}P_i) \in G_k,$$ 
           where $\bar{P}_iP_{\gamma}P_i$ is a loop starting from $v_i$. 
           By Lemma \ref{p in component}, there is a loop $P_{g} \subset C_i$ starting from $v_i$ such that $$g=\mathrm{Lab}(\bar{P}_iP_{\gamma}P_i)=\mathrm{Lab}(P_{g}),$$ which implies that $g \in \mathrm{Lab}(C_i,v_i)$ and then $\gamma \in L_i$.
      \end{proof}

    \section{Proof of the main theorem}\label{proof}
    
    In this section, we employ Wilton's construction and Markus-Epstein's algorithm to establish Theorem \ref{main2}.

    Let $F_r=\langle X \mid -\rangle$ be a free group of rank $r>1$, $G=\langle Y \mid R \rangle$ a group and 
    $$H= \langle h_1, h_2,\cdots,h_n\rangle \leq F_r \ast G.$$ 
   For any finite set $\{\gamma_1,\cdots,\gamma_m\} \subset (F_r\ast G)\setminus H$, we pick a finite connected subgraph $\Gamma$ of $\mathrm{Cay}(F_r \ast G,H)$ as follows: 
    \begin{align}
        \Gamma=\bigcup_{i=1}^n p_i\bigcup_{j=1}^m p'_j \tag{$\ast$} \label{gamma}\subset \mathrm{Cay}(F_r \ast G,H),
    \end{align}
   where all the loops $p_i$ and paths $p'_j$ have the same initial point $H \cdot 1$. and the labels of $p_i$ and $p'_j$ are given by $h_i$ and $\gamma_j$ respectively. Note that $\Gamma$ is an $F_r\ast G$-based labeled graph because it embeds in the $F_r\ast G$-based graph $\mathrm{Cay}(F_r \ast G,H)$. Moreover, any $F_r$-component of $\Gamma$ is $F_r$-based and any $G$-component of $\Gamma$ is $G$-based by Lemma \ref{free product based}. 

    \begin{lem}\label{H equals Lab}
    $\mathrm{Lab}(\Gamma, H \cdot 1 )=\mathrm{Lab}(\mathrm{Cay}(F_r \ast G,H), H \cdot 1)=H$.
    \end{lem}
    \begin{proof}
       The last ``=" follows from Lemma \ref{loop lem in relative Cayley graph}. Since $\Gamma$ is the subgraph of $\mathrm{Cay}(F_r \ast G,H)$, each loop in $\mathrm{Lab}(\Gamma, H \cdot 1 )$ is again a loop in $\mathrm{Lab}(\mathrm{Cay}(F_r \ast G,H), H \cdot 1)$, that is, 
          $$\mathrm{Lab}(\Gamma, H \cdot 1 ) \leq \mathrm{Lab}(\mathrm{Cay}(F_r \ast G,H), H \cdot 1).$$
       Conversely, since the labels of the loops $p_i(1 \leq  i \leq n)$ are given by the generators $h_i$ of $H$, we have $H \leq \mathrm{Lab}(\Gamma, H \cdot 1)$ clearly.
    \end{proof}

    We can describe the $\A \cup \mS$-separability of $H$ via the embedding property of $\Gamma$. 

    \begin{prop}\label{prop}
        If there exists a subgroup $H_0\leq F_r \ast G$ with prime index $p$ for $k=|V(\Gamma)|+2$ as in Lemma \ref{Jordan for p}, such that $(\Gamma, H\cdot 1)$ can be embedded in $\mathrm{Cay}(F_r \ast G,H_0,H_0\cdot 1)$ and there is a $\gamma \in F_r \ast  G$ acting on $\V(\mathrm{Cay}(F_r \ast G,H_0))$ nontrivially and moving less than $k$ vertices, then $H$ is $\A \cup \mS$-separable in $F_r \ast G$.
    \end{prop}
    \begin{proof}
        Following the notation above, we will show that there exists a surjection $f$ from $F_r \ast G$ to an alternating or symmetric group such that $f(\gamma_j) \notin f(H)$ for all $1 \leq j \leq m$.  

        By the hypothesis, there exists a subgroup $H_0\leq F_r \ast G$ with prime index $p$ as in Lemma \ref{Jordan for p} such that $(\Gamma, H\cdot 1)$ embeds in $\mathrm{Cay}(F_r \ast G,H_0,H_0\cdot 1)$. 
        Note that $$|\V(\mathrm{Cay}(F_r \ast G,H_0))|=p$$ and the action of $F_r \ast G$ on $\V(\mathrm{Cay}(F_r \ast G,H_0))$ is transitive and there exists $\gamma\in F_r \ast G$ moves less than $k$ vertices, then we have a homomorphism 
        $$f: F_r \ast G\to S_p$$ with $f(F_r \ast G)=S_p$ or $A_p$ by Lemma \ref{Jordan for p}. Moreover, since  for any $j=1, \ldots, m$,
        $$\gamma_j\notin H=\mathrm{Lab}(\mathrm{Cay}(F_r \ast G,H),H \cdot 1),$$ the path $p'_j$ is not closed in $\Gamma$ and hence not closed again in $\mathrm{Cay}(F_r \ast G,H_0)$. Therefore, the action of $\gamma_j$ on $\V(\mathrm{Cay}(F_r \ast G,H_0,H_0\cdot 1))$ does not fix the base point $H_0 \cdot 1$ while the action of $H$ does, it implies $f(\gamma_j) \notin f(H)$ for any $1 \leq j \leq m$.
    \end{proof}

    \begin{lem}\label{no cover component 1}
        If $H \cap F_r^{\gamma}\neq 1$ has infinite index in $F_r^{\gamma}$ for some $\gamma \in F_r \ast G$, then there is an $F_r$-component in $\Gamma$ which is not a cover of $F_r$.
    \end{lem}

    \begin{proof}
       By Lemma \ref{H equals Lab}, we have $H=\mathrm{Lab}(\Gamma, v_0)$ with base point $v_0:=H \cdot 1$. Then combining Kurosh subgroup theorem with Lemma \ref{intersection}, there is an $X$-component $C$ of $\Gamma$ such that $$\gamma\mathrm{Lab}(C,v)\gamma^{-1}=H \cap F_r^{\gamma},$$
       where $\gamma=\mathrm{Lab}(P)$ for $P$ an approach path from the base point $v_0$ to the base point $v$ of $C$. Hence, the index
       $$[F_r:\mathrm{Lab}(C,v)]=[F_r^\gamma : \gamma\mathrm{Lab}(C,v)\gamma^{-1}]=[F_r^\gamma : H \cap F_r^\gamma]=\infty.$$
       It follows that the $X$-component $C$ is not a finite cover of $F_r$. In fact, $C$ is not a cover of $F_r$ because it is finite.
    \end{proof}

    We now proceed to prove Theorem \ref{main2}.

   \begin{thmbis}{main2} 
  Let $G$ be a LERF group and let $H\leq F_r\ast G$ be a finitely generated subgroup. Then $H$ is $\A \cup \mS$-separable in $F_r\ast G$ if one of the following holds:
    \begin{enumerate}
        \item $H \cap F_r^{\gamma}=1$ for every $\gamma \in F_r \ast G$, or
        \item $H \cap F_r^{\gamma} \neq 1$ has infinite index in $F_r^{\gamma}$ for some $\gamma \in F_r \ast G$.
    \end{enumerate}
   \end{thmbis}
 \begin{proof}
    By Proposition \ref{prop}, we need to find a subgroup  $H_0\leq F_r\ast G$, or equivalently, a cover $$\widetilde{\Gamma}:=\mathrm{Cay}(F_r\ast G,H_0,H_0\cdot 1),$$
    whose index $[F_r\ast G: H_0]=|\V(\widetilde\Gamma)|$ is a sufficiently large prime $p$, such that $\Gamma$ in Eq. (\ref{gamma}) embeds into the finite cover $\widetilde\Gamma$ and there exists an element $\gamma\in  F_r\ast G$ moving less than a certain number of vertices in $\widetilde{\Gamma}$.

    We construct $\widetilde\Gamma$ in the following four steps.
  
         \textbf{Step 1. Embedding of each monochromatic component into a finite cover.}
        
        Since $\Gamma$ is finite, there must be finitely many monochromatic components in $\Gamma$, we denote them by $\Gamma_1,\ldots,\Gamma_\ell$, respectively. 
        
        If hypothesis (1) holds, either there is no $F_r$-component in $\Gamma$ (i.e., $\Gamma$ is $Y$-monochromatic), or every $F_r$-component in $\Gamma$ is a finite tree. Then for the former case, we denote the base point $H\cdot 1$ by $\Gamma_1$ and re-index the only $Y$-monochromatic component as $\Gamma_2$. For the latter case, we assume such an $F_r$-component is $\Gamma_1$ which is not a cover of $F_r$. 
        
        If hypothesis (2) holds, then by Lemma \ref{no cover component 1}, there is an $F_r$-component in $\Gamma$ which is not a cover of $F_r$. Without loss of generality, we again assume such an $F_r$-component is $\Gamma_1$. 
        
        Recall that each finite connected labeled component $\Gamma_i$ is $F_r$- or $G$-based, then $\Gamma_i$ can be embedded in a cover of $F_r$ or $G$ by Lemma \ref{subgraphs of Cayley}. Moreover,  since both $F_r$ and $G$ are LERF, every $\Gamma_i$ can be embedded in a finite cover $\widetilde\Gamma_i$ of $F_r$ or $G$ respectively by Theorem \ref{scott}.

        \textbf{Step 2. Construction of $\Gamma^\ast$.}

        We now consider the pushout of the embedding $\alpha_2: \bigcup_{i\geq 2}\Gamma_i\hookrightarrow \bigcup_{i\geq 2}\widetilde\Gamma_i$ as follows:
         $$\begin{tikzcd}
           \mathop{\bigcup}\limits_{i \geq 2} \Gamma_i  \arrow [hookrightarrow, r,"\alpha_1"] \arrow[hookrightarrow, d, "\alpha_2"]  &  \Gamma \arrow[d] \\
           \mathop{\bigcup}\limits_{i \geq 2} \widetilde\Gamma_i\arrow[r]  &  \Gamma^\ast
           \end{tikzcd}$$
         Note that every $\Gamma_i$ embeds in $\widetilde\Gamma_i$ and  $\Gamma_i$ is a monochromatic component of $\Gamma$, so there are no foldings occur between distinct $\widetilde\Gamma_i$, it follows $\Gamma_i$ embeds in $\Gamma^\ast$. Moreover, every monochromatic component of $\Gamma^\ast$ is $F_r$- or $G$-based, hence $\Gamma^\ast$ is $F_r\ast G$-based by Lemma \ref{free product based}. (See Figure \ref{figure2}.)

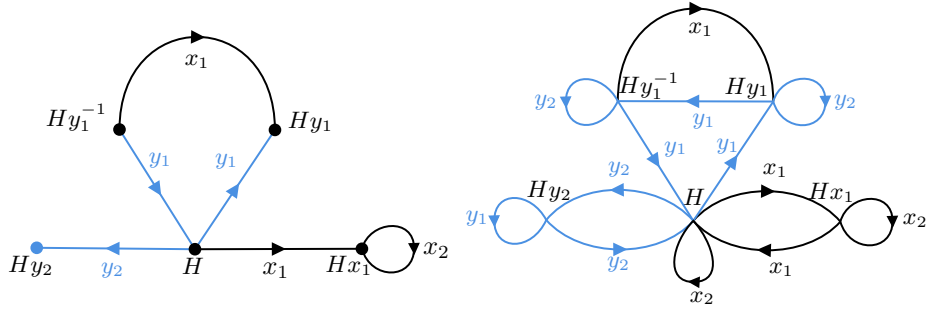
\begin{figure}

\begin{center}
\tikzset{every picture/.style={line width=0.75pt}} 

\begin{tikzpicture}[x=0.75pt,y=0.75pt,yscale=-1,xscale=1]

\draw [color={rgb, 255:red, 74; green, 144; blue, 226 }  ,draw opacity=1 ]   (80.23,157.79) -- (118.3,218.21) ;
\draw [shift={(101.29,191.22)}, rotate = 237.78] [fill={rgb, 255:red, 74; green, 144; blue, 226 }  ,fill opacity=1 ][line width=0.08]  [draw opacity=0] (7.14,-3.43) -- (0,0) -- (7.14,3.43) -- cycle    ;
\draw [color={rgb, 255:red, 74; green, 144; blue, 226 }  ,draw opacity=1 ]   (118.3,218.21) -- (158.65,157.99) ;
\draw [shift={(140.59,184.94)}, rotate = 123.82] [fill={rgb, 255:red, 74; green, 144; blue, 226 }  ,fill opacity=1 ][line width=0.08]  [draw opacity=0] (7.14,-3.43) -- (0,0) -- (7.14,3.43) -- cycle    ;
\draw [color={rgb, 255:red, 74; green, 144; blue, 226 }  ,draw opacity=1 ]   (38.5,217.49) -- (118.3,218.21) ;
\draw [shift={(73.1,217.8)}, rotate = 0.52] [fill={rgb, 255:red, 74; green, 144; blue, 226 }  ,fill opacity=1 ][line width=0.08]  [draw opacity=0] (7.14,-3.43) -- (0,0) -- (7.14,3.43) -- cycle    ;
\draw [shift={(38.5,217.49)}, rotate = 0.52] [color={rgb, 255:red, 74; green, 144; blue, 226 }  ,draw opacity=1 ][fill={rgb, 255:red, 74; green, 144; blue, 226 }  ,fill opacity=1 ][line width=0.75]      (0, 0) circle [x radius= 2.68, y radius= 2.68]   ;
\draw    (118.3,218.21) -- (202.5,217.99) ;
\draw [shift={(202.5,217.99)}, rotate = 359.85] [color={rgb, 255:red, 0; green, 0; blue, 0 }  ][fill={rgb, 255:red, 0; green, 0; blue, 0 }  ][line width=0.75]      (0, 0) circle [x radius= 2.68, y radius= 2.68]   ;
\draw [shift={(164.2,218.09)}, rotate = 179.85] [fill={rgb, 255:red, 0; green, 0; blue, 0 }  ][line width=0.08]  [draw opacity=0] (7.14,-3.43) -- (0,0) -- (7.14,3.43) -- cycle    ;
\draw [shift={(118.3,218.21)}, rotate = 359.85] [color={rgb, 255:red, 0; green, 0; blue, 0 }  ][fill={rgb, 255:red, 0; green, 0; blue, 0 }  ][line width=0.75]      (0, 0) circle [x radius= 2.68, y radius= 2.68]   ;
\draw    (80.23,157.79) .. controls (80.27,95.63) and (158.36,95.24) .. (158.65,157.99) ;
\draw [shift={(158.65,157.99)}, rotate = 89.74] [color={rgb, 255:red, 0; green, 0; blue, 0 }  ][fill={rgb, 255:red, 0; green, 0; blue, 0 }  ][line width=0.75]      (0, 0) circle [x radius= 2.68, y radius= 2.68]   ;
\draw [shift={(123.45,111.27)}, rotate = 181.76] [fill={rgb, 255:red, 0; green, 0; blue, 0 }  ][line width=0.08]  [draw opacity=0] (7.14,-3.43) -- (0,0) -- (7.14,3.43) -- cycle    ;
\draw [shift={(80.23,157.79)}, rotate = 270.04] [color={rgb, 255:red, 0; green, 0; blue, 0 }  ][fill={rgb, 255:red, 0; green, 0; blue, 0 }  ][line width=0.75]      (0, 0) circle [x radius= 2.68, y radius= 2.68]   ;
\draw [color={rgb, 255:red, 0; green, 0; blue, 0 }  ,draw opacity=1 ]   (202.5,217.99) .. controls (206.41,203.69) and (228.82,201.74) .. (228.82,217.76) .. controls (228.82,233.78) and (207.11,234.17) .. (202.5,217.99) -- cycle ;
\draw  [fill={rgb, 255:red, 0; green, 0; blue, 0 }  ,fill opacity=1 ] (228.91,219.96) -- (226.58,215.63) -- (230.9,215.47) -- cycle ;

\draw (93.52,168.26) node [anchor=north west][inner sep=0.75pt]  [font=\small] [align=left] {\textcolor[rgb]{0.29,0.56,0.89}{$\displaystyle y_{1}$}};
\draw (127.02,168.76) node [anchor=north west][inner sep=0.75pt]  [font=\small] [align=left] {\textcolor[rgb]{0.29,0.56,0.89}{$\displaystyle y_{1}$}};
\draw (41.65,143.52) node [anchor=north west][inner sep=0.75pt]  [font=\small] [align=left] {$\displaystyle Hy_{1}^{-1}$};
\draw (163.65,147.52) node [anchor=north west][inner sep=0.75pt]  [font=\small] [align=left] {$\displaystyle Hy_{1}$};
\draw (111.65,118.52) node [anchor=north west][inner sep=0.75pt]  [font=\small] [align=left] {$\displaystyle x_{1}$};
\draw (22.65,220.02) node [anchor=north west][inner sep=0.75pt]  [font=\small] [align=left] {$\displaystyle Hy_{2}$};
\draw (110.15,221.02) node [anchor=north west][inner sep=0.75pt]  [font=\small] [align=left] {$\displaystyle H$};
\draw (182.65,220.02) node [anchor=north west][inner sep=0.75pt]  [font=\small] [align=left] {$\displaystyle Hx_{1}$};
\draw (150.83,223.5) node [anchor=north west][inner sep=0.75pt]   [align=left] {$\displaystyle x_{1}$};
\draw (231.83,213.5) node [anchor=north west][inner sep=0.75pt]   [align=left] {$\displaystyle x_{2}$};
\draw (69.52,223.76) node [anchor=north west][inner sep=0.75pt]  [font=\small] [align=left] {\textcolor[rgb]{0.29,0.56,0.89}{$\displaystyle y_{2}$}};

\draw [color={rgb, 255:red, 74; green, 144; blue, 226 }  ,draw opacity=1 ]   (295.47,203.46) .. controls (312.93,223.15) and (352.22,223.32) .. (369.64,203.87) ;
\draw [shift={(336.4,218.24)}, rotate = 179.5] [fill={rgb, 255:red, 74; green, 144; blue, 226 }  ,fill opacity=1 ][line width=0.08]  [draw opacity=0] (7.14,-3.43) -- (0,0) -- (7.14,3.43) -- cycle    ;
\draw    (369.64,203.87) .. controls (386.96,183.86) and (426.39,183.86) .. (443.81,204.29) ;
\draw [shift={(410.94,189.08)}, rotate = 181.43] [fill={rgb, 255:red, 0; green, 0; blue, 0 }  ][line width=0.08]  [draw opacity=0] (7.14,-3.43) -- (0,0) -- (7.14,3.43) -- cycle    ;
\draw    (369.64,203.87) .. controls (387.1,223.56) and (426.39,223.73) .. (443.81,204.29) ;
\draw [shift={(401.64,218.52)}, rotate = 2.27] [fill={rgb, 255:red, 0; green, 0; blue, 0 }  ][line width=0.08]  [draw opacity=0] (7.14,-3.43) -- (0,0) -- (7.14,3.43) -- cycle    ;
\draw    (331.56,143.46) .. controls (331.6,81.3) and (409.69,80.91) .. (409.98,143.66) ;
\draw [shift={(374.78,96.94)}, rotate = 181.76] [fill={rgb, 255:red, 0; green, 0; blue, 0 }  ][line width=0.08]  [draw opacity=0] (7.14,-3.43) -- (0,0) -- (7.14,3.43) -- cycle    ;
\draw [color={rgb, 255:red, 74; green, 144; blue, 226 }  ,draw opacity=1 ]   (331.56,143.46) -- (409.98,143.66) ;
\draw [shift={(365.47,143.54)}, rotate = 0.14] [fill={rgb, 255:red, 74; green, 144; blue, 226 }  ,fill opacity=1 ][line width=0.08]  [draw opacity=0] (7.14,-3.43) -- (0,0) -- (7.14,3.43) -- cycle    ;
\draw [color={rgb, 255:red, 74; green, 144; blue, 226 }  ,draw opacity=1 ]   (331.56,143.46) .. controls (323.9,128.97) and (305.34,128.19) .. (305.34,143.82) .. controls (305.34,159.45) and (324.25,159.45) .. (331.56,143.46) -- cycle ;
\draw  [color={rgb, 255:red, 74; green, 144; blue, 226 }  ,draw opacity=1 ][fill={rgb, 255:red, 74; green, 144; blue, 226 }  ,fill opacity=1 ] (305.43,146.02) -- (303.1,141.69) -- (307.42,141.53) -- cycle ;

\draw [color={rgb, 255:red, 74; green, 144; blue, 226 }  ,draw opacity=1 ]   (409.98,143.66) .. controls (413.89,129.36) and (436.3,127.4) .. (436.3,143.43) .. controls (436.3,159.45) and (414.59,159.84) .. (409.98,143.66) -- cycle ;
\draw  [color={rgb, 255:red, 74; green, 144; blue, 226 }  ,draw opacity=1 ][fill={rgb, 255:red, 74; green, 144; blue, 226 }  ,fill opacity=1 ] (436.39,145.63) -- (434.06,141.3) -- (438.38,141.14) -- cycle ;

\draw [color={rgb, 255:red, 74; green, 144; blue, 226 }  ,draw opacity=1 ]   (331.56,143.46) -- (369.64,203.87) ;
\draw [shift={(352.62,176.88)}, rotate = 237.78] [fill={rgb, 255:red, 74; green, 144; blue, 226 }  ,fill opacity=1 ][line width=0.08]  [draw opacity=0] (7.14,-3.43) -- (0,0) -- (7.14,3.43) -- cycle    ;
\draw [color={rgb, 255:red, 74; green, 144; blue, 226 }  ,draw opacity=1 ]   (369.64,203.87) -- (409.98,143.66) ;
\draw [shift={(391.92,170.61)}, rotate = 123.82] [fill={rgb, 255:red, 74; green, 144; blue, 226 }  ,fill opacity=1 ][line width=0.08]  [draw opacity=0] (7.14,-3.43) -- (0,0) -- (7.14,3.43) -- cycle    ;
\draw [color={rgb, 255:red, 74; green, 144; blue, 226 }  ,draw opacity=1 ]   (295.47,203.46) .. controls (312.35,182.89) and (351.56,183.28) .. (369.64,203.87) ;
\draw [shift={(327.44,188.38)}, rotate = 358.74] [fill={rgb, 255:red, 74; green, 144; blue, 226 }  ,fill opacity=1 ][line width=0.08]  [draw opacity=0] (7.14,-3.43) -- (0,0) -- (7.14,3.43) -- cycle    ;
\draw [color={rgb, 255:red, 74; green, 144; blue, 226 }  ,draw opacity=1 ]   (295.47,203.46) .. controls (287.81,188.97) and (269.25,188.18) .. (269.25,203.81) .. controls (269.25,219.44) and (288.16,219.44) .. (295.47,203.46) -- cycle ;
\draw  [color={rgb, 255:red, 74; green, 144; blue, 226 }  ,draw opacity=1 ][fill={rgb, 255:red, 74; green, 144; blue, 226 }  ,fill opacity=1 ] (269.34,206.02) -- (267.01,201.69) -- (271.33,201.53) -- cycle ;
\draw [color={rgb, 255:red, 0; green, 0; blue, 0 }  ,draw opacity=1 ]   (369.64,203.87) .. controls (358.92,214.8) and (354.72,235.12) .. (370.12,234.73) .. controls (385.53,234.34) and (379.58,214.8) .. (369.64,203.87) -- cycle ;
\draw  [fill={rgb, 255:red, 0; green, 0; blue, 0 }  ,fill opacity=1 ] (367.92,234.59) -- (372.46,232.71) -- (372.19,237.02) -- cycle ;

\draw [color={rgb, 255:red, 0; green, 0; blue, 0 }  ,draw opacity=1 ]   (443.81,204.29) .. controls (447.72,189.99) and (470.13,188.04) .. (470.13,204.06) .. controls (470.13,220.08) and (448.42,220.47) .. (443.81,204.29) -- cycle ;

\draw  [fill={rgb, 255:red, 0; green, 0; blue, 0 }  ,fill opacity=1 ] (470.21,206.26) -- (467.89,201.93) -- (472.21,201.77) -- cycle ;

\draw (365.65,102.52) node [anchor=north west][inner sep=0.75pt]  [font=\small] [align=left] {$x_1$};
\draw (367.74,148.29) node [anchor=north west][inner sep=0.75pt]  [font=\small] [align=left] {\textcolor[rgb]{0.29,0.56,0.89}{$y_1$}};
\draw (353.52,162.26) node [anchor=north west][inner sep=0.75pt]  [font=\small] [align=left] {\textcolor[rgb]{0.29,0.56,0.89}{$y_1$}};
\draw (379.82,160.93) node [anchor=north west][inner sep=0.75pt]  [font=\small] [align=left] {\textcolor[rgb]{0.29,0.56,0.89}{$y_1$}};
\draw (288.07,138.92) node [anchor=north west][inner sep=0.75pt]  [font=\small] [align=left] {\textcolor[rgb]{0.29,0.56,0.89}{$y_2$}};
\draw (439.18,138.72) node [anchor=north west][inner sep=0.75pt]  [font=\small] [align=left] {\textcolor[rgb]{0.29,0.56,0.89}{$y_2$}};
\draw (403.24,173.62) node [anchor=north west][inner sep=0.75pt]  [font=\small] [align=left] {$x_1$};
\draw (406.34,223.51) node [anchor=north west][inner sep=0.75pt]  [font=\small] [align=left] {$x_1$};
\draw (473.8,198.95) node [anchor=north west][inner sep=0.75pt]  [font=\small] [align=left] {$x_2$};
\draw (253.36,197.49) node [anchor=north west][inner sep=0.75pt]  [font=\small] [align=left] {\textcolor[rgb]{0.29,0.56,0.89}{$y_1$}};
\draw (325.23,172.55) node [anchor=north west][inner sep=0.75pt]  [font=\small] [align=left] {\textcolor[rgb]{0.29,0.56,0.89}{$y_2$}};
\draw (324.73,220.56) node [anchor=north west][inner sep=0.75pt]  [font=\small] [align=left] {\textcolor[rgb]{0.29,0.56,0.89}{$y_2$}};
\draw (366.44,238.08) node [anchor=north west][inner sep=0.75pt]  [font=\small] [align=left] {$x_2$};
\draw (284.01,182.16) node [anchor=north west][inner sep=0.75pt]  [font=\small] [align=left] {$Hy_2$};
\draw (362.95,186.31) node [anchor=north west][inner sep=0.75pt]  [font=\small] [align=left] {$H$};
\draw (425.92,183.41) node [anchor=north west][inner sep=0.75pt]  [font=\small] [align=left] {$Hx_1$};
\draw (330.24,127.5) node [anchor=north west][inner sep=0.75pt]  [font=\small] [align=left] {$Hy_1^{-1}$};
\draw (383.79,130.01) node [anchor=north west][inner sep=0.75pt]  [font=\small] [align=left] {$Hy_1$};

\end{tikzpicture}

\end{center}
\caption{An example for $\Gamma$ and $\Gamma^\ast$ when $X=\{x_1,x_2\}$, $Y=\{ y_1,y_2 \}$. We take $H=\langle y_1x_1^{-1}y_1, x_1x_2x_1^{-1}\rangle$ and $\gamma_1=y_2$.}\label{figure2}
\end{figure}
       \textbf{Step 3.  Modify $\Gamma_1$ to obtain a precover $\Gamma'$ of $F_r\ast G$.}

         Denote $|\V(\Gamma^\ast)|=k$. Since the labeled graph $\Gamma_1$ is $F_r$-based and not a cover, it is not an $X^\ast$-saturated graph by Lemma \ref{subgraphs of Cayley}. From Wilton's construction in Lemma \ref{alternating embedding}, we know that after attaching some graphs $W_n$ and $V_s$, the graph $\Gamma_1$ can embed in a cover $\widetilde\Gamma_1$ of $F_r$ with 
         $$|\V(\widetilde\Gamma_1)|=|\V(\Gamma_1)|+n+4.$$ We now take a sufficiently large prime $p$ for $k+5$ as in Lemma \ref{Jordan for p}, and let $n=p-k-4$. We construct the pushout of the embedding $\alpha_2':\Gamma_1 \hookrightarrow \widetilde\Gamma_1$ as following:
         $$\begin{tikzcd}
           \Gamma_1  \arrow[hookrightarrow, r, "\alpha'_1"] \arrow[hookrightarrow, d, "\alpha'_2"]  &  \Gamma^\ast \arrow[d] \\
           \widetilde\Gamma_1 \arrow[r]  &  \Gamma'
           \end{tikzcd}$$
           
        Notice that $\Gamma_1$ embeds in $\widetilde\Gamma_1$ and $\widetilde\Gamma_1$ is monochromatic, there are no foldings occur between $\widetilde\Gamma_1$ and $\Gamma^\ast$, then both $\Gamma^\ast$ and $\widetilde\Gamma_1$ embed in $\Gamma'$. Since $\Gamma^\ast$ is $F_r \ast G$-based and $\widetilde\Gamma_1$ is $F_r$-based, $\Gamma'$ is $F_r\ast G$-based by Lemma \ref{free product based} and every monochromatic component is a cover. So $\Gamma'$ is a precover of $F_r\ast G$.
         
         \textbf{Step 4. Construction of the final cover $\widetilde{\Gamma}$.}
         
         By Lemma \ref{embedding of precover}, the precover $\Gamma'$ can embed in a cover $\widetilde{\Gamma}$ of $F_r \ast G$ without adding new vertices, so $\Gamma$ itself embeds in the cover $\widetilde{\Gamma}$ with
         $$|\mathrm{V}(\widetilde{\Gamma})|=p=k+n+4.$$
         Since $x_2$ acts trivially on $\V(W_n)$ and acts nontrivially on $\V(V_s)$, there must be less than $k+5$ vertices (lying in $\Gamma^\ast \cup V_s$ ) in $\widetilde{\Gamma}$ moved by $x_2$.  Therefore, the cover $\widetilde{\Gamma}$ is precisely what we need, and the proof is complete.
    \end{proof}

    \noindent\textbf{Acknowledgements.}  The authors would like to thank Qian Chen very much for his comments. Thanks are also extended to the editor and the anonymous referee for their time.


\end{document}